\documentclass[10pt]{article}

\usepackage{amsmath,amsbsy,amssymb,amsthm}
\usepackage{a4wide}
\usepackage{graphicx,psfrag}
\usepackage{subfigure}

\newtheorem{Theorem}{Theorem}
\newtheorem{Definition}{Definition}

\newtheorem{Lemma}{Lemma}

\def\a{{\alpha(t)}}
\def\Da{{\alpha'(t)}}
\def\t{\tau}
\def\DI{{_a\mathbb{D}_t^\a}}
\def\DII{{_aD_t^\a}}
\def\DIII{{_a\mathcal{D}_t^\a}}
\def\DIR{{_t\mathbb{D}_b^\a}}
\def\DIIR{{_tD_b^\a}}
\def\DIIIR{{_t\mathcal{D}_b^\a}}
\def\C{\binom{n-\a}{p}}
\def\D{\binom{1-\a}{p}}
\def\DS{\displaystyle}

\begin{document}

\title{Caputo--Hadamard fractional derivatives of variable order}

\author{Ricardo Almeida\\
\texttt{ricardo.almeida@ua.pt}}

\date{\em{Center for Research and Development in Mathematics and Applications (CIDMA),
Department of Mathematics, University of Aveiro, 3810--193 Aveiro, Portugal}}

\maketitle


\begin{abstract}

In this paper we present three types of Caputo--Hadamard derivatives of variable fractional order, and study the relations between them.
An approximation formula for each fractional operator, using integer-order derivatives only, is obtained, and an estimation for the error is given. At the end  we compare the exact fractional derivative of a concrete example with some numerical approximations.

\bigskip

\noindent \textbf{Keywords}: fractional calculus, variable fractional order, Caputo fractional derivative, Hadamard fractional derivative, expansion formulas.

\smallskip

\noindent \textbf{Mathematics Subject Classification 2010}: 26A33, 33F05.
\end{abstract}


\section{Introduction}

At the same time ordinary calculus was developed for integer-order derivatives in the seventeenth century, L'H\^{o}pital and Leibniz  wondered about the notion of derivative of order $n=1/2$. 
Two centuries later, with the works of Fourier, Riemann, Liouville, Hadamard, Gr\"{u}nwald, etc., numerous definitions trying to generalize the notion of ordinary derivative were developed.
This was primarily a theoretical study for mathematicians, with several notions of fractional operators appearing and their properties well studied.
In the past decades, with the discovery that processes like conservation mass, viscoelasticity, nanotechnology, signal processing, and several other applications in engineering are better described by fractional derivatives, these discoveries have gained a great importance.
Fractional derivatives are nonlocal concepts, and  for this reason may be more suitable to translate natural phenomena. In our days, one can find several books and journals dedicated exclusively to fractional calculus theory, not only on the subject of mathematics but also physics, engineering, economics, applied sciences, etc.

Due the complexity of dealing with fractional operators, we find different numerical approaches to solve the desired problems. One of those available methods consists in approximating the fractional derivative by an expansion that depends on integer-order derivatives only \cite{Atanackovic1,Pooseh4,Pooseh5}. With this technique in hand, given any problem involving fractional operators, we can simply replace  them with the given  expansion, obtaining a new problem that depends on integer-order derivatives only. After that, we can apply any known technique to solve it. The main advantage of this procedure is that we do not need higher-order derivatives in order to have a good approximation, in opposite to other methods e.g. the approximation for the Riemann--Liouville fractional derivative \cite{samko},
$${_aD_t^\alpha}x(t)=\sum_{n=0}^\infty \binom{\alpha}{n}\frac{(t-a)^{n-\alpha}}{\Gamma(n+1-\alpha)}x^{(n)}(t)$$
and for the Hadamard fractional derivative we have  \cite{Butzer2}
$${_0^HD_t^\alpha}x(t)=\sum_{n=0}^\infty S(\alpha,n)t^n x^{(n)}(t),$$
where $S(\alpha,n)$ is the Stirling function.

\section{Fractional calculus of variable order}
\label{sec:FC}

In this section, we present three new types of fractional operators, which combine the Caputo with the Hadamard fractional derivatives. The order of the derivative is a function $\a$, that depends on time, and when it is constant we prove that the three definitions coincide.  To start, we review some concepts on fractional derivative operators with constant order \cite{Machado,Miller,samko}. Let $\alpha$ be a real in the interval $(0,1)$ and $x$ be a real valued function with domain $[a,b]$. The Caputo fractional derivative of $x$ of order $\alpha$ is given by
$${_a^CD_t^\alpha}x(t)={_aD_t^\alpha}[x(t)-x(a)],$$
where ${_aD_t^\alpha}$ stands for the Riemann-Liouville fractional derivative:
$${_aD_t^\alpha}x(t)=\frac{1}{\Gamma(1-\alpha)}\frac{d}{dt}\int_a^t(t-\t)^{-\alpha}x(\t)d\t.$$
If $x$ is differentiable, using integration by parts and then differentiating the integral, we obtain an equivalent definition for the Caputo fractional derivative:
$${_a^CD_t^\alpha}x(t)=\frac{1}{\Gamma(1-\alpha)}\int_a^t(t-\t)^{-\alpha}x'(\t)d\t.$$
With respect to the Hadamard fractional derivative, we have the following definition:
$${_a^HD_t^\alpha}x(t)=\frac{t}{\Gamma(1-\alpha)}\frac{d}{dt}\int_a^t \left(\ln\frac{t}{\tau}\right)^{-\alpha}\frac{x(\tau)}{\tau}d\tau.$$
For example, for $\gamma>0$, if we take the two functions $x(t)=(t-a)^\gamma$ and $y(t)=\left(\ln\frac{t}{a}\right)^\gamma$, then
$${_a^CD_t^\alpha}x(t)=\frac{\Gamma(\gamma+1)}{\Gamma(\gamma+1-\alpha)} (t-a)^{\gamma-\alpha} \quad \mbox{and} \quad {_a^HD_t^\alpha}y(t)=  \frac{\Gamma(\gamma+1)}{\Gamma(\gamma+1-\alpha)} \left(\ln\frac{t}{a}\right)^{\gamma-\alpha}.$$
One natural extension of these two concepts is to define the Caputo--Hadamard fractional derivative of order $\alpha$ \cite{Gambo5,Jarad5}:
$${_{\,\,\,\,\,a}^{CH}D_t^\alpha}x(t)=\frac{t}{\Gamma(1-\alpha)}\frac{d}{dt}\int_a^t \left(\ln\frac{t}{\tau}\right)^{-\alpha}\frac{x(\tau)-x(a)}{\tau}d\tau,$$
or in an equivalent way,
$${_{\,\,\,\,\,a}^{CH}D_t^\alpha}x(t)=\frac{1}{\Gamma(1-\alpha)}\int_a^t \left(\ln\frac{t}{\tau}\right)^{-\alpha}x'(\tau)d\tau.$$
We extend the previous notions by considering the order of the derivative a real valued function $\alpha:[a,b]\to(0,1)$. This is a very recent direction of research, and was studied for the first time in \cite{SamkoRoss} with respect to  the Riemann--Liouville fractional derivative. Since these fractional operators are nonlocal, and contain memory about the past dynamic, it is natural to consider the order of the derivative also variable, depending on the process. In fact, several applications were found and it is nowadays a  growing subject that has attracted the attention of a vast community   \cite{Coimbra,Coimbra2,Dzherbashyan,Lin,Od,Ramirez2,Sun}. In this paper we deal with a combined Caputo--Hadamard fractional derivative with fractional variable order. Three different types of operators are given.

\begin{Definition} Let $a,b$ be two reals with $0<a<b$, and $x:[a,b]\to\mathbb R$ be a function. The left Caputo--Hadamard fractional derivative of order $\a$
\begin{enumerate}
\item type 1 is defined by
$$\DI x(t)=\frac{1}{\Gamma(1-\a)}\int_a^t\left(\ln\frac{t}{\t}\right)^{-\a}x'(\t)d\t;$$
\item type 2 is defined by
$$\DII x(t)=\frac{t}{\Gamma(1-\a)}\frac{d}{dt}\left(\int_a^t\left(\ln\frac{t}{\t}\right)^{-\a}\frac{x(\t)-x(a)}{\t}d\t\right);$$
\item type 3 is defined by
$$\DIII x(t)=t\frac{d}{dt}\left(\frac{1}{\Gamma(1-\a)}\int_a^t\left(\ln\frac{t}{\t}\right)^{-\a}\frac{x(\t)-x(a)}{\t}d\t\right).$$
\end{enumerate}
\end{Definition}

Analogous definitions for the right fractional operators are given:

\begin{Definition} Let $a,b$ be two reals with $0<a<b$, and $x:[a,b]\to\mathbb R$ be a function. The right Caputo--Hadamard fractional derivative of order $\a$
\begin{enumerate}
\item type 1 is defined by
$$\DIR x(t)=\frac{-1}{\Gamma(1-\a)}\int_t^b\left(\ln\frac{\t}{t}\right)^{-\a}x'(\t)d\t;$$
\item type 2 is defined by
$$\DIIR x(t)=\frac{-t}{\Gamma(1-\a)}\frac{d}{dt}\left(\int_t^b\left(\ln\frac{\t}{t}\right)^{-\a}\frac{x(\t)-x(b)}{\t}d\t\right);$$
\item type 3 is defined by
$$\DIIIR x(t)=-t\frac{d}{dt}\left(\frac{1}{\Gamma(1-\a)}\int_t^b\left(\ln\frac{\t}{t}\right)^{-\a}\frac{x(\t)-x(b)}{\t}d\t\right).$$
\end{enumerate}
\end{Definition}

We will see that these definitions do not coincide. To start, we prove the following Lemma.

\begin{Lemma}\label{LemmaEx}
Let $\gamma>0$ and consider the function $x(t)=\left(\ln\frac{t}{a}\right)^\gamma$. Then
\begin{enumerate}
\item $\DI x(t) =\DS\frac{\Gamma(\gamma+1)}{\Gamma(\gamma+1-\a)} \left(\ln\frac{t}{a}\right)^{\gamma-\a}$.
\item $\DII x(t) =\DS\frac{\Gamma(\gamma+1)}{\Gamma(\gamma+1-\a)} \left(\ln\frac{t}{a}\right)^{\gamma-\a}$

$\DS\quad -\frac{t\Da\Gamma(\gamma+1)}{\Gamma(\gamma+2-\a)} \left(\ln\frac{t}{a}\right)^{\gamma+1-\a} \left[\ln \left(\ln\frac{t}{a}\right)+\psi(1-\a)-\psi(\gamma+2-\a)\right]$.
\item $\DIII x(t) =\DS\frac{\Gamma(\gamma+1)}{\Gamma(\gamma+1-\a)} \left(\ln\frac{t}{a}\right)^{\gamma-\a}$

$\DS\quad -\frac{t\Da\Gamma(\gamma+1)}{\Gamma(\gamma+2-\a)} \left(\ln\frac{t}{a}\right)^{\gamma+1-\a} \left[\ln \left(\ln\frac{t}{a}\right)-\psi(\gamma+2-\a)\right]$.
\end{enumerate}
\end{Lemma}

\begin{proof} We only prove the first one; the others are proven in a similar way. First, observe that
$$\begin{array}{ll}
\DI x(t)&=\DS\frac{1}{\Gamma(1-\a)}\int_a^t\left(\ln\frac{t}{\t}\right)^{-\a}\frac{\gamma}{\t}\left(\ln\frac{\t}{a}\right)^{\gamma-1}d\t\\
&=\DS\frac{\gamma}{\Gamma(1-\a)}\left(\ln\frac{t}{a}\right)^{-\a}\int_a^t\left(1-\frac{\ln\frac{\t}{a}}{\ln\frac{t}{a}}\right)^{-\a}
\left(\ln\frac{\t}{a}\right)^{\gamma-1}\frac{d\t}{\t}.
\end{array}$$
If we proceed with the change of variables
$$\ln\frac{\t}{a}=s\ln\frac{t}{a},$$
we get
$$\begin{array}{ll}
\DI x(t) &=\DS \frac{\gamma}{\Gamma(1-\a)}\left(\ln\frac{t}{a}\right)^{-\a}\int_0^1 (1-s)^{-\a}s^{\gamma-1}\left(\ln\frac{t}{a}\right)^{\gamma-1}
\left(\ln\frac{t}{a}\right) ds\\
&=\DS \frac{\gamma}{\Gamma(1-\a)}\left(\ln\frac{t}{a}\right)^{\gamma-\a}\int_0^1 (1-s)^{-\a}s^{\gamma-1}ds\\
&=\DS \frac{\gamma}{\Gamma(1-\a)}\left(\ln\frac{t}{a}\right)^{\gamma-\a}B(1-\a,\gamma),
\end{array}$$
where $B(\cdot,\cdot)$ is the beta function.
Using the relation
$$B(1-\a,\gamma)=\frac{\Gamma(1-\a)\Gamma(\gamma)}{\Gamma(\gamma+1-\a)},$$
we prove the desired formula.
\end{proof}

In a similar way we have the following:

\begin{Lemma}
Let $\gamma>0$ and consider the function $x(t)=\left(\ln\frac{b}{t}\right)^\gamma$. Then
\begin{enumerate}
\item $\DIR x(t) =\DS\frac{\Gamma(\gamma+1)}{\Gamma(\gamma+1-\a)} \left(\ln\frac{b}{t}\right)^{\gamma-\a}$.
\item $\DIIR x(t) =\DS\frac{\Gamma(\gamma+1)}{\Gamma(\gamma+1-\a)} \left(\ln\frac{b}{t}\right)^{\gamma-\a}$

$\DS\quad +\frac{t\Da\Gamma(\gamma+1)}{\Gamma(\gamma+2-\a)} \left(\ln\frac{b}{t}\right)^{\gamma+1-\a} \left[\ln \left(\ln\frac{b}{t}\right)+\psi(1-\a)-\psi(\gamma+2-\a)\right]$.
\item $\DIIIR x(t) =\DS\frac{\Gamma(\gamma+1)}{\Gamma(\gamma+1-\a)} \left(\ln\frac{b}{t}\right)^{\gamma-\a}$

$\DS\quad +\frac{t\Da\Gamma(\gamma+1)}{\Gamma(\gamma+2-\a)} \left(\ln\frac{b}{t}\right)^{\gamma+1-\a} \left[\ln \left(\ln\frac{b}{t}\right)-\psi(\gamma+2-\a)\right]$.
\end{enumerate}
\end{Lemma}

\begin{Theorem}\label{relations} The following relations hold:
$$\DII x(t)=\DI x(t)+\frac{t\Da}{\Gamma(2-\a)}\int_a^t\left(\ln\frac{t}{\t}\right)^{1-\a}x'(\t)\left[\frac{1}{1-\a}-\ln\left(\ln\frac{t}{\t}\right)\right]d\t$$
and
$$\DIII x(t)=\DII x(t)+\frac{t\Da\psi(1-\a)}{\Gamma(1-\a)}\int_a^t\left(\ln\frac{t}{\t}\right)^{-\a}\frac{x(\t)-x(a)}{\t}d\t.$$
\end{Theorem}

\begin{proof} Starting with the formula
$$\DII x(t)=\frac{t}{\Gamma(1-\a)}\frac{d}{dt}\left(\int_a^t\left(\ln\frac{t}{\t}\right)^{-\a}\frac{x(\t)-x(a)}{\t}d\t\right),$$
and integrating by parts choosing
$$u'(\t)=\left(\ln\frac{t}{\t}\right)^{-\a}\frac{1}{\t} \quad \mbox{and} \quad v(\t)=x(\t)-x(a),$$
we obtain
$$\DII x(t)=\frac{t}{\Gamma(1-\a)}\frac{d}{dt}\left(\frac{1}{1-\a}\int_a^t\left(\ln\frac{t}{\t}\right)^{1-\a}x'(\t)d\t\right).$$
Differentiating the product, we prove the formula. The second one follows immediately from the definition.
\end{proof}

From this result, for an arbitrary function $x$, we see that these three definitions coincide
$$\DI x(t)\equiv \DII x(t)\equiv \DIII x(t),$$
 only when the order $\a$ or the function $x$ are constant.

For what concerns the right fractional operators, we have the two following relations.
$$\DIIR x(t)=\DIR x(t)+\frac{t\Da}{\Gamma(2-\a)}\int_t^b\left(\ln\frac{\t}{t}\right)^{1-\a}x'(\t)\left[\frac{1}{1-\a}-\ln\left(\ln\frac{\t}{t}\right)\right]d\t$$
and
$$\DIIIR x(t)=\DIIR x(t)-\frac{t\Da\psi(1-\a)}{\Gamma(1-\a)}\int_t^b\left(\ln\frac{\t}{t}\right)^{-\a}\frac{x(\t)-x(b)}{\t}d\t.$$

\begin{Theorem} Let $x$ be of class $C^1$. Then
$$\DI x(t)=\DII x(t)=\DIII x(t)=0,$$
at $t=a$.
\end{Theorem}

\begin{proof} For what concerns $\DI x(t)$, we have
$$\left|\DI x(t)\right|\leq\frac{\|x'\|}{\Gamma(1-\a)}\int_a^t\left(\ln\frac{t}{\t}\right)^{-\a}\frac{1}{\t}\cdot\t d\t.$$
Integrating by parts, we get
$$\left|\DI x(t)\right|\leq\frac{\|x'\|}{\Gamma(2-\a)}
\left[a\left(\ln\frac{t}{a}\right)^{1-\a}+\int_a^t\left(\ln\frac{t}{\t}\right)^{1-\a}d\t\right],$$
which vanishes at $t=a$. To prove that $\DII x(t)=0$ at $t=a$, using Theorem  \ref{relations}, is enough to prove that
$$\int_a^t\left(\ln\frac{t}{\t}\right)^{1-\a}d\t= \int_a^t\left(\ln\frac{t}{\t}\right)^{1-\a}\ln\left(\ln\frac{t}{\t}\right)d\t=0,$$
for $t=a$. With respect to the first integral, is obvious. For the second one, let
$$f(\t)=\left(\ln\frac{t}{\t}\right)^{1-\a}\ln\left(\ln\frac{t}{\t}\right), \quad \t\in[a,t[.$$
Since $f(\t)\to0$ as $\t\to t$, we can extend continuously $f$ to the closed interval $[a,t]$ by letting $f(t)=0$. Finding the extremals for $f$, we prove that for all  $\t\in[a,t]$,
$$\left|f(\t)\right|\leq \max\left\{\left(\ln\frac{t}{a}\right)^{1-\a}\left|\ln\left(\ln\frac{t}{a}\right)\right|, \frac{1}{e(1-\a)}\right\}.$$
With this we prove the second part. The last one is clear, using the second relation in Theorem  \ref{relations} and integration by parts.
\end{proof}

%

\section{Expansion formulas for the Caputo--Hadamard fractional derivatives}\label{sec:expansion}

Define the sequence $(x_k)_{k \in\mathbb N}$ recursively by the formula
$$x_1(t)=tx'(t) \quad \mbox{and} \quad x_{k+1}(t)= tx'_k(t), \, k \in\mathbb N.$$
Also, for $k\in\mathbb N$, define the quantities
$$\begin{array}{ll}
A_k & =\DS \frac{1}{\Gamma(k+1-\a)}\left[1+\sum_{p=n-k+1}^N \frac{\Gamma(\a-n+p)}{\Gamma(\a-k)(p-n+k)!}  \right],\\
B_k & =  \DS\frac{\Gamma(\a-n+k)}{\Gamma(1-\a)\Gamma(\a)(k-n)!},
\end{array}$$
and the function
$$V_k(t)=\int_{a}^{t}\left(\ln\frac{\t}{a}\right)^k x'(\t)d\t.$$

\begin{Theorem}\label{teo1} Let $x:[a,b]\to\mathbb R$ be a function of class $C^{n+1}$, for $n\in\mathbb N$, and fix $N\in\mathbb N$ with $N \geq n$.
Then,
$$\DI x(t) =\DS\sum_{k=1}^{n}A_k \left(\ln\frac{t}{a}\right)^{k-\a}x_k(t)+\sum_{k=n}^N B_k \left(\ln\frac{t}{a}\right)^{n-k-\a} V_{k-n}(t)+E(t),$$
with
$$ E(t)\leq \frac{(t-a)\exp((n-\a)^2+n-\a)}{\Gamma(n+1-\a)N^{n-\a}(n-\a)}\left(\ln\frac{t}{a}\right)^{n-\a}\max_{\t\in[a,t]}\left|x'_N(\t)\right|.$$
\end{Theorem}

\begin{proof} Integrating by parts the integral
$$\DI x(t)=\frac{1}{\Gamma(1-\a)}\int_a^t\left(\ln\frac{t}{\t}\right)^{-\a}\frac{1}{\t}x_1(\t)d\t,$$
with
$$u'(\t)=\left(\ln\frac{t}{\t}\right)^{-\a}\frac{1}{\t} \quad \mbox{and} \quad  v(\t)=x_1(\t),$$
we get
$$\DI x(t)= \frac{1}{\Gamma(2-\a)}\left(\ln\frac{t}{a}\right)^{1-\a}x_1(a)+\frac{1}{\Gamma(2-\a)}\int_a^t\left(\ln\frac{t}{\t}\right)^{1-\a}\frac{1}{\t}x_2(\t)d\t.$$
If we proceed integrating by parts $n-1$ more times, the following formula is obtained:
$$\DI x(t) =\DS\sum_{k=1}^n \frac{1}{\Gamma(k+1-\a)}\left(\ln\frac{t}{a}\right)^{k-\a}x_k(a)+\frac{1}{\Gamma(n+1-\a)}
\int_a^t\left(\ln\frac{t}{\t}\right)^{n-\a}\frac{1}{\t}x_{n+1}(\t)d\t.$$
By the Taylor's Theorem, we obtain the sum
$$\begin{array}{ll}
\DS\left(\ln\frac{t}{\t}\right)^{n-\a}&=\DS\left(\ln\frac{t}{a}\right)^{n-\a}\left(1-\frac{\ln\frac{\t}{a}}{\ln\frac{t}{a}}\right)^{n-\a}\\
&=\DS\left(\ln\frac{t}{a}\right)^{n-\a}\sum_{p=0}^N \C (-1)^p \frac{\left(\ln\frac{\t}{a}\right)^p}{\left(\ln\frac{t}{a}\right)^p}+E_1(t),
\end{array}$$
where
$$E_1(t)=\left(\ln\frac{t}{a}\right)^{n-\a}\sum_{p=N+1}^\infty \C (-1)^p \frac{\left(\ln\frac{\t}{a}\right)^p}{\left(\ln\frac{t}{a}\right)^p}$$
and
$$\C (-1)^p=\frac{\Gamma(\a-n+p)}{\Gamma(\a-n)p!}.$$
Using this relation, we get the new formula
$$\begin{array}{ll}
\DI x(t)&=\DS\sum_{k=1}^n \frac{1}{\Gamma(k+1-\a)}\left(\ln\frac{t}{a}\right)^{k-\a}x_k(a)\\
& \quad \DS +\frac{1}{\Gamma(n+1-\a)}\left(\ln\frac{t}{a}\right)^{n-\a} \sum_{p=0}^N \frac{\Gamma(\a-n+p)}{\Gamma(\a-n)p!\left(\ln\frac{t}{a}\right)^p}\\
& \quad \DS \times\int_a^t \left(\ln\frac{\t}{a}\right)^p\frac{1}{\t}x_{n+1}(\t) d\t+E(t),
\end{array}$$
where
$$E(t)=\frac{1}{\Gamma(n+1-\a)}\int_a^t E_1(t)\frac{1}{\t}x_{n+1}(\t)d\t.$$
If we split the sum into the first term $p=0$ and the remaining ones $p=1\ldots N$, and use integration by parts taking
$$u(\t)=\left(\ln\frac{\t}{a}\right)^p\quad \mbox{and} \quad v'(\t)=\frac{1}{\t}x_{n+1}(\t)=x'_n(\t),$$
we get
$$\begin{array}{ll}
\DI x(t)&=\DS\sum_{k=1}^{n-1} \frac{1}{\Gamma(k+1-\a)}\left(\ln\frac{t}{a}\right)^{k-\a}x_k(a)+A_n\left(\ln\frac{t}{a}\right)^{n-\a}x_n(t)\\
& \quad \DS +\frac{1}{\Gamma(n-\a)}\left(\ln\frac{t}{a}\right)^{n-1-\a}\sum_{p=1}^N \frac{\Gamma(\a-n+p)}{\Gamma(\a+1-n)(p-1)!\left(\ln\frac{t}{a}\right)^{p-1}}\\
& \quad \DS \times\int_a^t \left(\ln\frac{\t}{a}\right)^{p-1}\frac{1}{\t}x_{n}(\t) d\t+E(t).
\end{array}$$
Observe that
$$\frac{1}{\t}x_{n}(\t)=x'_{n-1}(\t).$$
Repeating this procedure, i.e., splinting the second sum  (first term $p=k$ plus the remaining ones $p=k+1\ldots N$) and integrating by parts the integral that appears in the sum $p=k+1\ldots N$, we obtain the desired the formula.
For the error, observe that for $\t\in[a,t]$, we have
$$0\leq  \frac{\left(\ln\frac{\t}{a}\right)^p}{\left(\ln\frac{t}{a}\right)^p}\leq 1.$$
Thus,
$$\begin{array}{ll}
\left|E_1(t)\right|&\DS\leq \left(\ln\frac{t}{a}\right)^{n-\a}\sum_{p=N+1}^\infty \frac{\exp((n-\a)^2+n-\a)}{p^{n+1-\a}}\\
&\DS\leq \left(\ln\frac{t}{a}\right)^{n-\a}\int_N^\infty \frac{\exp((n-\a)^2+n-\a)}{p^{n+1-\a}}\, dp\\
&\DS= \left(\ln\frac{t}{a}\right)^{n-\a}\frac{\exp((n-\a)^2+n-\a)}{N^{n-\a}(n-\a)}.\\
\end{array}$$
The rest of the proof follows immediately.
\end{proof}

Observe that, for $N$ sufficiently large, we have the approximation
\begin{equation}\label{approx1}
\DI x(t) \approx\sum_{k=1}^{n}A_k \left(\ln\frac{t}{a}\right)^{k-\a}x_k(t)+\sum_{k=n}^N B_k \left(\ln\frac{t}{a}\right)^{n-k-\a} V_{k-n}(t).
\end{equation}

\begin{Theorem}\label{teo2} Let $x:[a,b]\to\mathbb R$ be a function of class $C^{n+1}$, for $n\in\mathbb N$, and fix $N\in\mathbb N$ with $N \geq n$.
Then,
$$\begin{array}{ll}
\DII x(t)& =\DS\sum_{k=1}^{n}A_k \left(\ln\frac{t}{a}\right)^{k-\a}x_k(t)+\sum_{k=n}^N B_k \left(\ln\frac{t}{a}\right)^{n-k-\a} V_{k-n}(t)
+\frac{t\Da}{\Gamma(2-\a)}\left(\ln\frac{t}{a}\right)^{1-\a}\\
&  \quad \DS \times\left[\left(\frac{1}{1-\a}-\ln\left(\ln\frac{t}{a}\right)\right)\sum_{p=0}^N\D\frac{(-1)^p}{ \left(\ln\frac{t}{a}\right)^{p}} V_{p}(t)\right.\\
& \quad \quad \quad \DS \left.+\sum_{p=0}^N\D(-1)^p\sum_{r=1}^N\frac{1}{r \left(\ln\frac{t}{a}\right)^{p+r}} V_{p+r}(t)\right]+E(t),
\end{array}$$
with
$$ \begin{array}{ll}
E(t)&\leq \DS \frac{(t-a)\exp((n-\a)^2+n-\a)}{\Gamma(n+1-\a)N^{n-\a}(n-\a)}\left(\ln\frac{t}{a}\right)^{n-\a}\max_{\t\in[a,t]}\left|x'_N(\t)\right|\\
&\quad \DS+\frac{(t-a)t\left|\Da\right|\exp((1-\a)^2+1-\a)}{\Gamma(2-\a)N^{1-\a}(1-\a)}\left(\ln\frac{t}{a}\right)^{1-\a}\max_{\t\in[a,t]}\left|x'(\t)\right|\\
&\quad \times \DS
\left|\frac{1}{1-\a}-\ln\left(\ln\frac{t}{a}\right)+\frac{(2t-a)\ln\frac{t}{a}}{N}\right|.
\end{array}$$
\end{Theorem}

\begin{proof} Recalling Theorem \ref{relations}, we have
$$\DII x(t)=\DI x(t)+\frac{t\Da}{\Gamma(2-\a)}\int_a^t\left(\ln\frac{t}{\t}\right)^{1-\a}x'(\t)\left[\frac{1}{1-\a}-\ln\left(\ln\frac{t}{\t}\right)\right]d\t,$$
and so we only need to expand the second term in the right side:
\begin{equation}\label{eq1}\frac{t\Da}{\Gamma(2-\a)}\int_a^t\left(\ln\frac{t}{\t}\right)^{1-\a}x'(\t)
\left[\frac{1}{1-\a}-\ln\left(\ln\frac{t}{\t}\right)\right]d\t.\end{equation}
On one hand, as was seen in the proof of Theorem \ref{teo1}, we have
$$\left(\ln\frac{t}{\t}\right)^{1-\a}=\left(\ln\frac{t}{a}\right)^{1-\a}\left[\sum_{p=0}^N \D (-1)^p \frac{\left(\ln\frac{\t}{a}\right)^p}{\left(\ln\frac{t}{a}\right)^p}+E_1(t)\right],$$
where
$$E_1(t)=\sum_{p=N+1}^\infty \D (-1)^p \frac{\left(\ln\frac{\t}{a}\right)^p}{\left(\ln\frac{t}{a}\right)^p}.$$
On the other hand,
$$\begin{array}{ll}
\DS\ln\left(\ln\frac{t}{\t}\right)&=\DS\ln\left(\ln\frac{t}{a}\right)+\ln\left(1-\frac{\ln\frac{\t}{a}}{\ln\frac{t}{a}}\right)\\
&=\DS\ln\left(\ln\frac{t}{a}\right)-\sum_{r=1}^N  \frac{\left(\ln\frac{\t}{a}\right)^r}{r\left(\ln\frac{t}{a}\right)^r} -E_2(t),
\end{array}$$
where
$$E_2(t)=\sum_{r=N+1}^\infty  \frac{\left(\ln\frac{\t}{a}\right)^r}{r\left(\ln\frac{t}{a}\right)^r}.$$
Substituting these two relations in Eq. \eqref{eq1}, we obtain
$$\frac{t\Da}{\Gamma(2-\a)}\left(\ln\frac{t}{a}\right)^{1-\a}$$
$$\times\left[\left(\frac{1}{1-\a}-\ln\left(\ln\frac{t}{a}\right)\right)
\sum_{p=0}^N\D\frac{(-1)^p}{\left(\ln\frac{t}{a}\right)^p} V_p(t)+\sum_{p=0}^N\D(-1)^p\sum_{r=1}^N\frac{1}{r\left(\ln\frac{t}{a}\right)^{p+r}} V_{p+r}(t)\right]$$
$$+\frac{t\Da}{\Gamma(2-\a)}\left(\ln\frac{t}{a}\right)^{1-\a}\times\left[\left(\frac{1}{1-\a}-\ln\left(\ln\frac{t}{a}\right)\right)
\int_a^t E_1(t)x'(\t)\,d\t+\int_a^t E_1(t)E_2(t)x'(\t)\,d\t\right].$$
We now determine the upper bound for the error. As was seen in proof of Theorem \ref{teo1}, we have
$$\left|E_1(t)\right|\leq \frac{\exp((1-\a)^2+1-\a)}{N^{1-\a}(1-\a)}.$$
On the other hand,
$$\int_a^t E_2(t) \, d\t=\sum_{r=N+1}^\infty \frac{1}{r\left(\ln\frac{t}{a}\right)^r}\int_a^t \t \cdot \left(\ln\frac{\t}{a}\right)^r\frac{d\t}{\t}.$$
Integrating by parts,
$$\begin{array}{ll}
\DS\int_a^t E_2(t) \, d\t & \DS = \sum_{r=N+1}^\infty \frac{1}{r(r+1)}\left[ t\ln\frac{t}{a}-
\int_a^t \frac{\left(\ln\frac{\t}{a}\right)^{r+1}}{\left(\ln\frac{t}{a}\right)^{r}}d\t\right]\\
& \DS \leq \sum_{r=N+1}^\infty \frac{1}{r(r+1)}\left[ t\ln\frac{t}{a}+\int_a^t \ln\frac{t}{a} d\t\right]\\
& \DS \leq (2t-a)\ln\frac{t}{a}\int_N^\infty \frac{1}{r^2}\,dr=\frac{(2t-a)\ln\frac{t}{a}}{N}.
\end{array}$$
Combining these relations, we prove the result.
\end{proof}

Finally, we have the expansion formula for the Caputo--Hadamard fractional derivative of order $\a$ type 3.

\begin{Theorem}\label{teo3} Let $x:[a,b]\to\mathbb R$ be a function of class $C^{n+1}$, for $n\in\mathbb N$, and fix $N\in\mathbb N$ with $N \geq n$.
Then,
$$\begin{array}{ll}
\DIII x(t)& =\DS\sum_{k=1}^{n}A_k \left(\ln\frac{t}{a}\right)^{k-\a}x_k(t)+\sum_{k=n}^N B_k \left(\ln\frac{t}{a}\right)^{n-k-\a} V_{k-n}(t)
+\frac{t\Da}{\Gamma(2-\a)}\left(\ln\frac{t}{a}\right)^{1-\a}\\
&  \quad \DS \times\left[\left(\psi(2-\a)-\ln\left(\ln\frac{t}{a}\right)\right)\sum_{p=0}^N\D\frac{(-1)^p}{ \left(\ln\frac{t}{a}\right)^{p}} V_{p}(t)\right.\\
& \quad \quad \quad \DS \left.+\sum_{p=0}^N\D(-1)^p\sum_{r=1}^N\frac{1}{r \left(\ln\frac{t}{a}\right)^{p+r}} V_{p+r}(t)\right]+E(t),
\end{array}$$
with
$$ \begin{array}{ll}
E(t)&\leq \DS \frac{(t-a)\exp((n-\a)^2+n-\a)}{\Gamma(n+1-\a)N^{n-\a}(n-\a)}\left(\ln\frac{t}{a}\right)^{n-\a}\max_{\t\in[a,t]}\left|x'_N(\t)\right|\\
&\quad \DS+\frac{(t-a)t\left|\Da\right|\exp((1-\a)^2+1-\a)}{\Gamma(2-\a)N^{1-\a}(1-\a)}\left(\ln\frac{t}{a}\right)^{1-\a}\max_{\t\in[a,t]}\left|x'(\t)\right|\\
&\quad \times \DS
\left|\psi(2-\a)-\ln\left(\ln\frac{t}{a}\right)+\frac{(2t-a)\ln\frac{t}{a}}{N}\right|.
\end{array}$$
\end{Theorem}

\begin{proof}  By Theorem \ref{relations}, we only need to expand the term
$$\frac{t\Da\psi(1-\a)}{\Gamma(1-\a)}\int_a^t\left(\ln\frac{t}{\t}\right)^{-\a}\frac{x(\t)-x(a)}{\t}d\t.$$
Integrating by parts, we get
$$\frac{t\Da\psi(1-\a)}{\Gamma(2-\a)}\int_a^t\left(\ln\frac{t}{\t}\right)^{1-\a}x'(\t)d\t.$$
Using Taylor's Theorem on the term $\left(\ln\frac{t}{\t}\right)^{1-\a}$, combining with Theorem \ref{teo2} and using the formula
$$\psi(1-\a)+\frac{1}{1-\a}=\psi(2-\a),$$
we prove the result.
\end{proof}

Similar formulas are obtained for the right Caputo--Hadamard fractional derivatives. Letting, for $k\in\mathbb N$,
$$\begin{array}{ll}
\overline A_k & =\DS \frac{(-1)^k}{\Gamma(k+1-\a)}\left[1+\sum_{p=n-k+1}^N \frac{\Gamma(\a-n+p)}{\Gamma(\a-k)(p-n+k)!}  \right],\\
\overline B_k & =  \DS\frac{-\Gamma(\a-n+k)}{\Gamma(1-\a)\Gamma(\a)(k-n)!},
\end{array}$$
and
$$\overline V_k(t)=\int_{t}^{b}\left(\ln\frac{b}{\t}\right)^k x'(\t)d\t,$$
we have the  next three approximation formulas:
$$\DIR x(t)\approx\DS\sum_{k=1}^{n}\overline A_k \left(\ln\frac{b}{t}\right)^{k-\a}x_k(t)+\sum_{k=n}^N \overline B_k \left(\ln\frac{b}{t}\right)^{n-k-\a} \overline V_{k-n}(t),$$
$$\DIIR x(t)\approx\DS\sum_{k=1}^{n}\overline A_k \left(\ln\frac{b}{t}\right)^{k-\a}x_k(t)+\sum_{k=n}^N \overline B_k \left(\ln\frac{b}{t}\right)^{n-k-\a} \overline V_{k-n}(t)+\frac{t\Da}{\Gamma(2-\a)}\left(\ln\frac{b}{t}\right)^{1-\a}$$
$$ \times\left[\left(\frac{1}{1-\a}-\ln\left(\ln\frac{b}{t}\right)\right)\sum_{p=0}^N\D\frac{(-1)^p}{ \left(\ln\frac{b}{t}\right)^{p}} \overline V_{p}(t)+\sum_{p=0}^N\D(-1)^p\sum_{r=1}^N\frac{1}{r \left(\ln\frac{b}{t}\right)^{p+r}}\overline V_{p+r}(t)\right],$$
and
$$\DIIIR x(t)\approx\DS\sum_{k=1}^{n}\overline A_k \left(\ln\frac{b}{t}\right)^{k-\a}x_k(t)+\sum_{k=n}^N \overline B_k \left(\ln\frac{b}{t}\right)^{n-k-\a} \overline V_{k-n}(t)+\frac{t\Da}{\Gamma(2-\a)}\left(\ln\frac{b}{t}\right)^{1-\a}$$
$$ \times\left[\left(\psi(2-\a)-\ln\left(\ln\frac{b}{t}\right)\right)\sum_{p=0}^N\D\frac{(-1)^p}{ \left(\ln\frac{b}{t}\right)^{p}} \overline V_{p}(t)+\sum_{p=0}^N\D(-1)^p\sum_{r=1}^N\frac{1}{r \left(\ln\frac{b}{t}\right)^{p+r}}\overline V_{p+r}(t)\right].$$

\section{Examples}
\label{sec:EX}

In this section we test the efficiency of the purposed method, by comparing the exact fractional derivative of the function
$$\overline x(t)=\ln t, \quad \mbox{for} \quad t\in[1,5],$$
with some numerical approximations. We fix the order $\a=t/20$, and for the approximations using Theorems \ref{teo1}, \ref{teo2} and \ref{teo3}, we take $n=1$ and vary $N\in\{10,20,30\}$. The error of approximation at the point $t_0$ is given by the absolute value of the difference between the exact and the approximation at $t=t_0$. Below we give the exact fractional derivatives of $\overline x$, obtained by Lemma \ref{LemmaEx}
$$\begin{array}{ll}
{_1\mathbb{D}_t^\a} \overline x(t) &=
\DS \frac{1}{\Gamma\left(2-\a\right)}(\ln t)^{1-\a}\\
{_1D_t^\a} \overline x(t) &=
\DS \frac{1}{\Gamma\left(2-\a\right)}(\ln t)^{1-\a}
 -\frac{t\Da}{\Gamma\left(3-\a\right)}(\ln t)^{2-\a}\\
 &\quad \DS\times\left[\ln(\ln t)+\psi\left(1-\a\right)
 -\psi\left(3-\a\right)\right],\\
{_1\mathcal{D}_t^\a} \overline x(t) &=
\DS \frac{1}{\Gamma\left(2-\a\right)}(\ln t)^{1-\a}
 -\frac{t\Da}{\Gamma\left(3-\a\right)}(\ln t)^{2-\a}\left[\ln(\ln t)-\psi\left(3-\a\right)\right].\\
\end{array}$$

\begin{figure}[!ht]
\begin{center}
\subfigure[${_1\mathbb{D}_t^\a}  \overline x(t)$]{\label{PlotI}\includegraphics[scale=0.25]{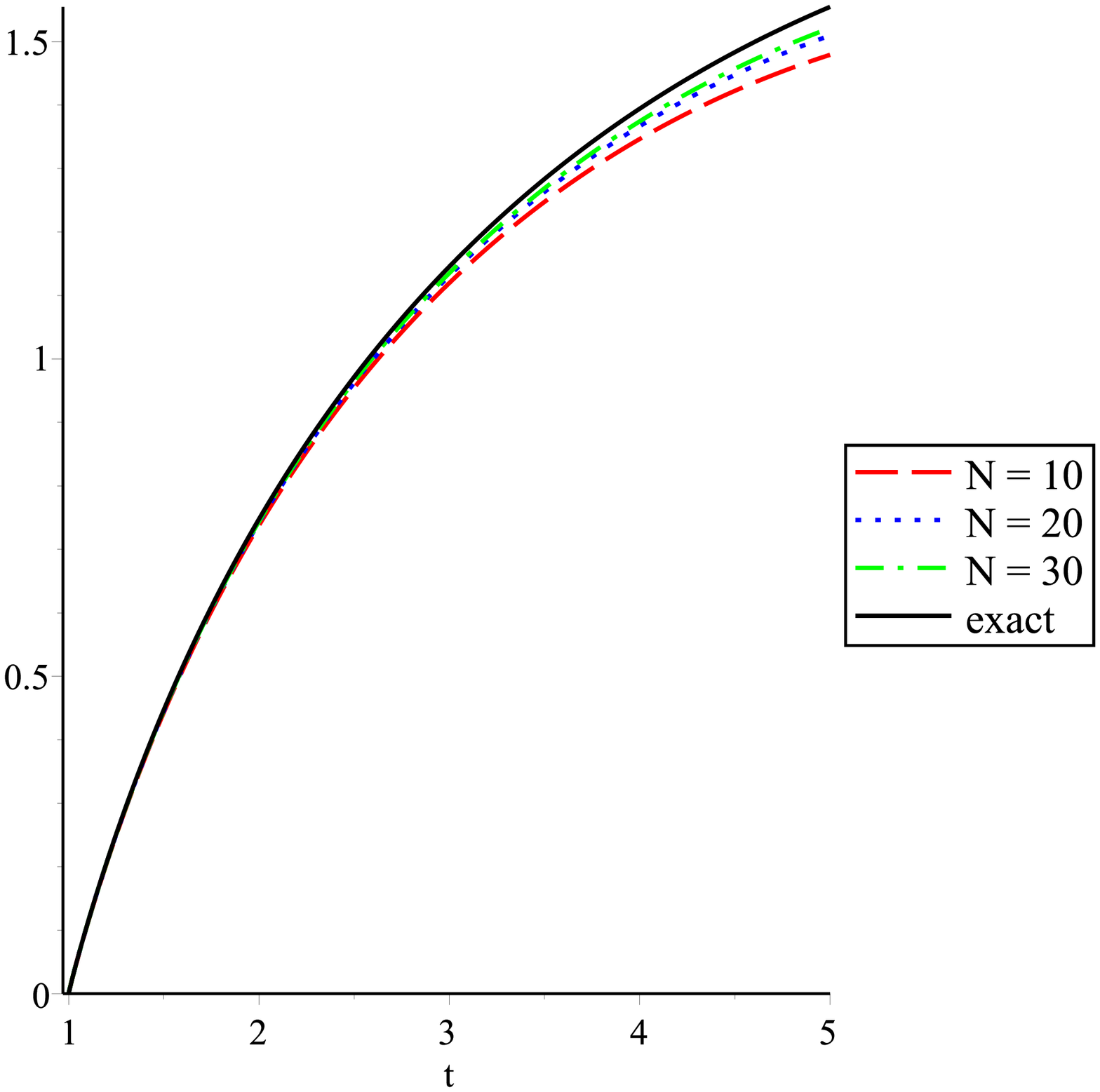}}\hspace{1cm}
\subfigure[Error]{\label{ExpExn}\includegraphics[scale=0.25]{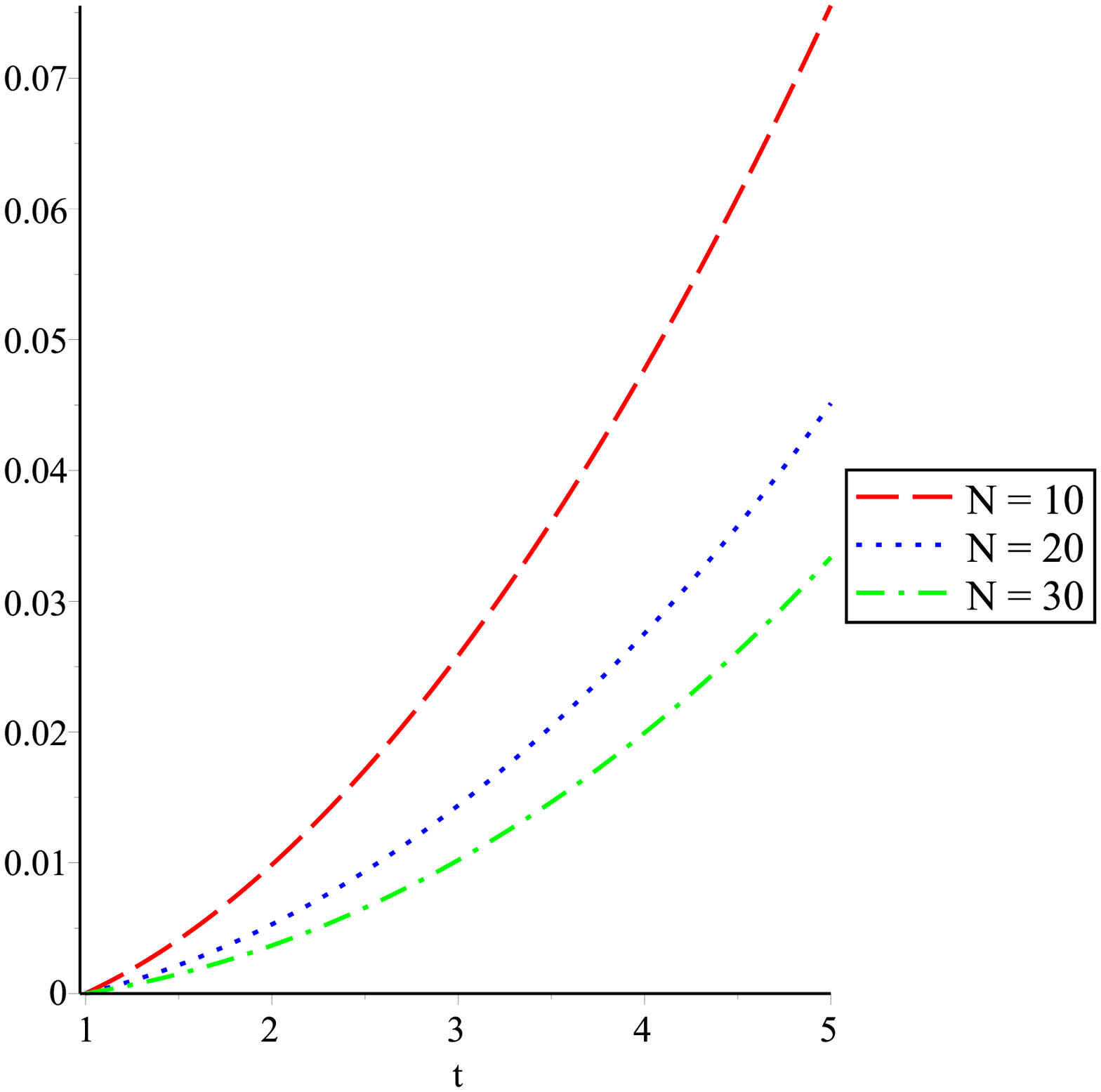}}
\end{center}
\caption{Analytic versus numerical approximations by Theorem \ref{teo1}.}
\label{IntExp}
\end{figure}

\begin{figure}[!ht]
\begin{center}
\subfigure[${_1D_t^\a}  \overline x(t)$]{\label{PlotI}\includegraphics[scale=0.25]{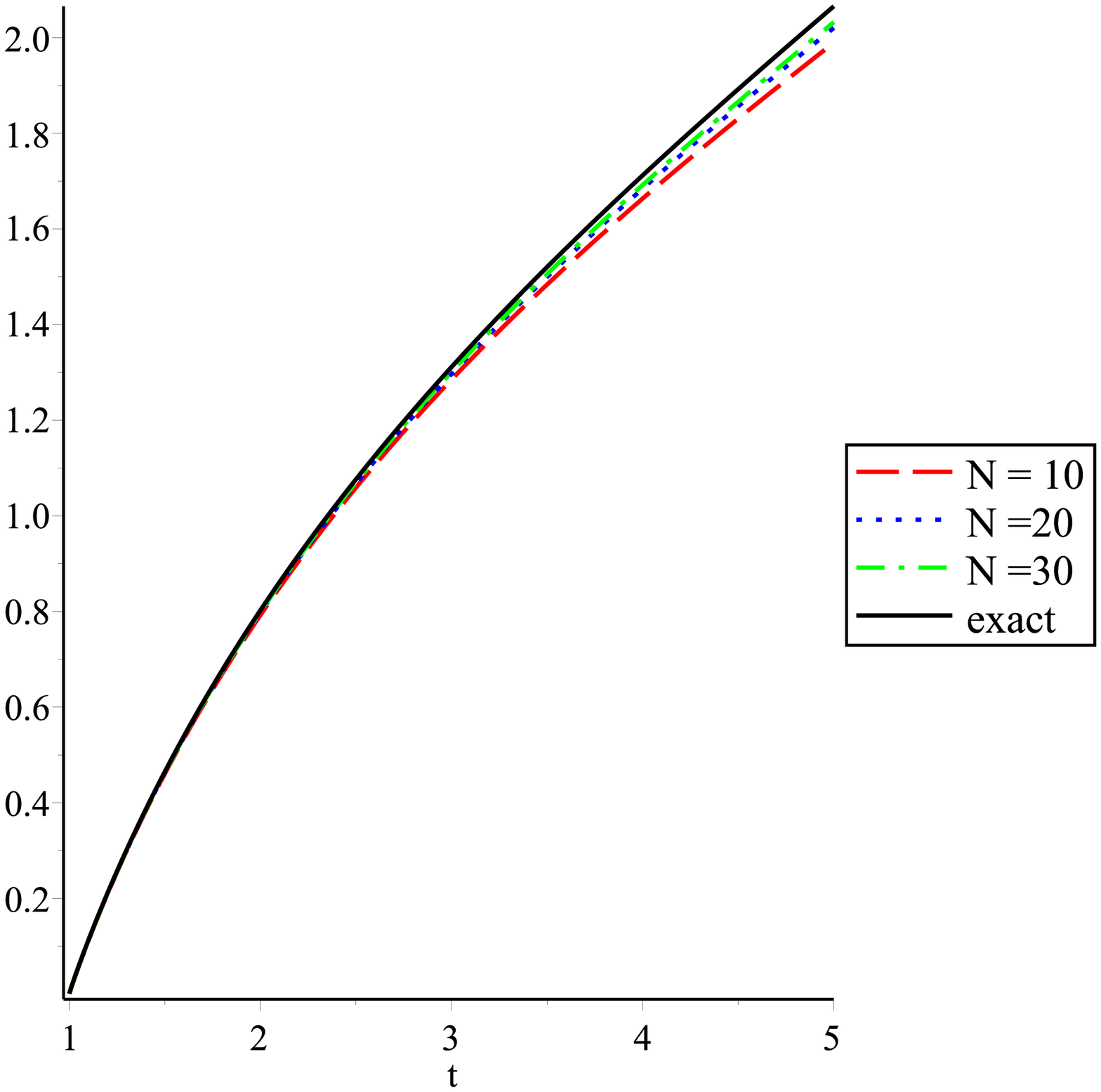}} \hspace{1cm}
\subfigure[Error]{\label{ExpExn}\includegraphics[scale=0.25]{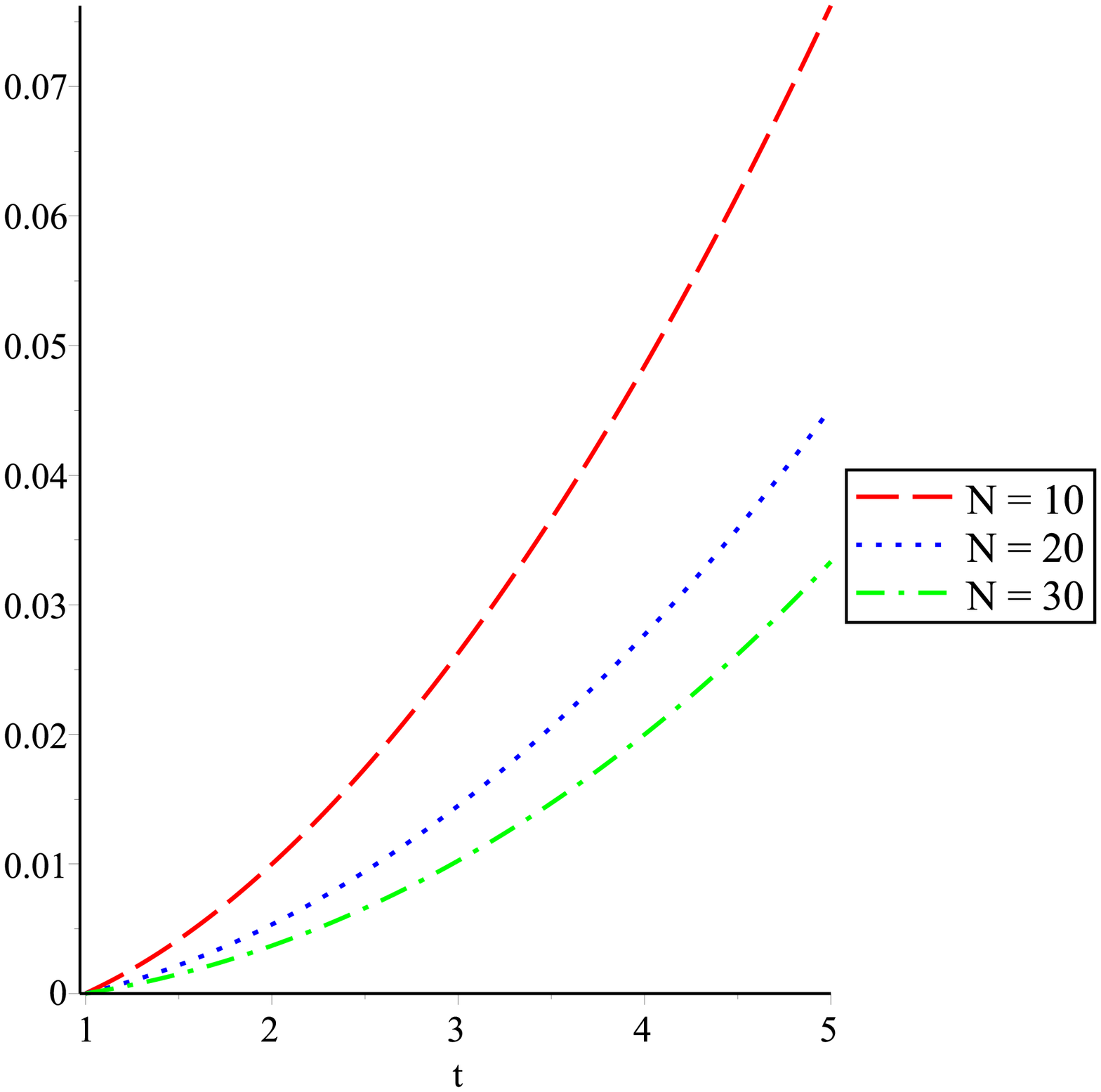}}
\end{center}
\caption{Analytic versus numerical approximations by Theorem \ref{teo2}.}
\label{IntExp}
\end{figure}

\begin{figure}[!ht]
\begin{center}
\subfigure[${_1\mathcal{D}_t^\a}  \overline x(t)$]{\label{PlotI}\includegraphics[scale=0.25]{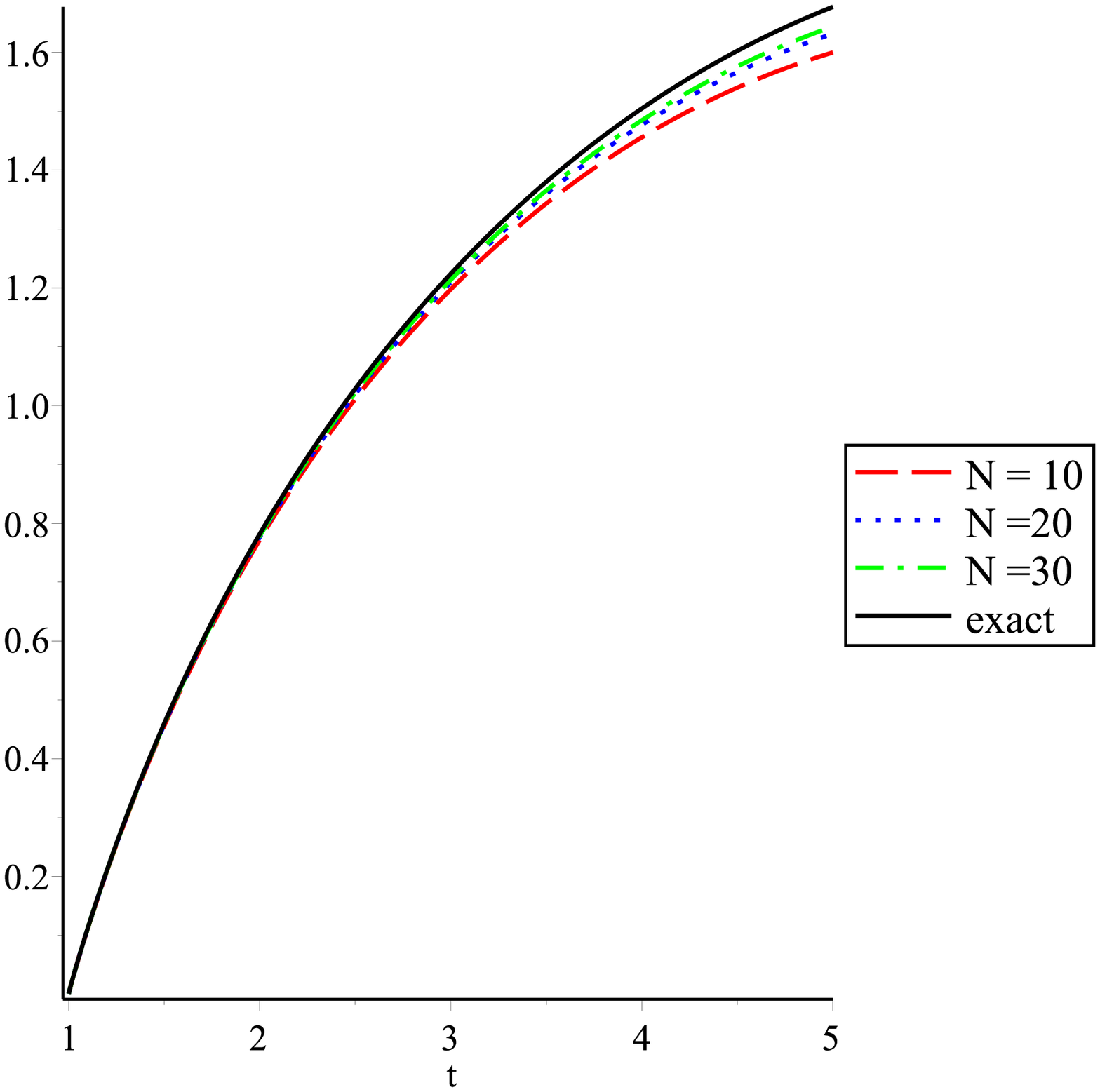}}\hspace{1cm}
\subfigure[Error]{\label{ExpExn}\includegraphics[scale=0.25]{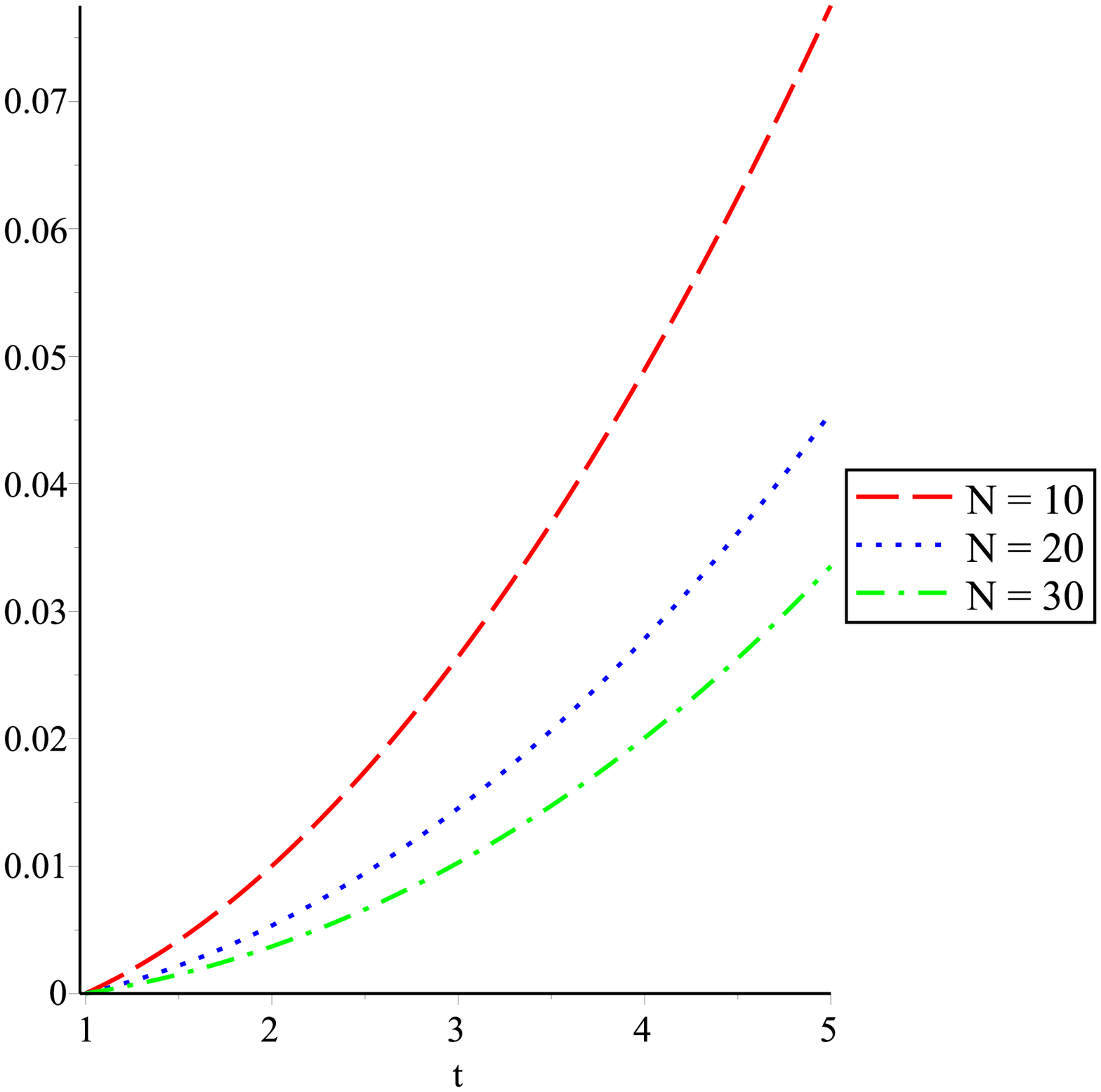}}
\end{center}
\caption{Analytic versus numerical approximations by Theorem \ref{teo3}.}
\label{IntExp}
\end{figure}

For the  right Caputo--Hadamard fractional derivatives, we take
$$\overline y(t)=\ln\frac5t, \quad \mbox{for} \quad t\in[1,5],$$
$n=1$ and $N\in\{2,4,6\}$. The fractional derivatives are in this case given by the expressions
$$\begin{array}{ll}
{_t\mathbb{D}_5^\a} \overline y(t) &=
\DS \frac{1}{\Gamma\left(2-\a\right)}\left(\ln \frac5t\right)^{1-\a}\\
{_tD_5^\a} \overline y(t) &=
\DS \frac{1}{\Gamma\left(2-\a\right)}\left(\ln \frac5t\right)^{1-\a}
 +\frac{t\Da}{\Gamma\left(3-\a\right)}\left(\ln \frac5t\right)^{2-\a}\\
 &\quad \DS\times\left[\ln\left(\ln \frac5t\right)+\psi\left(1-\a\right)
 -\psi\left(3-\a\right)\right],\\
{_t\mathcal{D}_5^\a} \overline y(t) &=
\DS \frac{1}{\Gamma\left(2-\a\right)}\left(\ln \frac5t\right)^{1-\a}
 +\frac{t\Da}{\Gamma\left(3-\a\right)}\left(\ln \frac5t\right)^{2-\a}\left[\ln\left(\ln \frac5t\right)-\psi\left(3-\a\right)\right].\\
\end{array}$$
The results are shown in Figure \ref{fig4} below.

\begin{figure}[!ht]
\begin{center}
\subfigure[${_t\mathbb{D}_5^\a}  \overline y(t)$]{\label{PlotI}\includegraphics[scale=0.25]{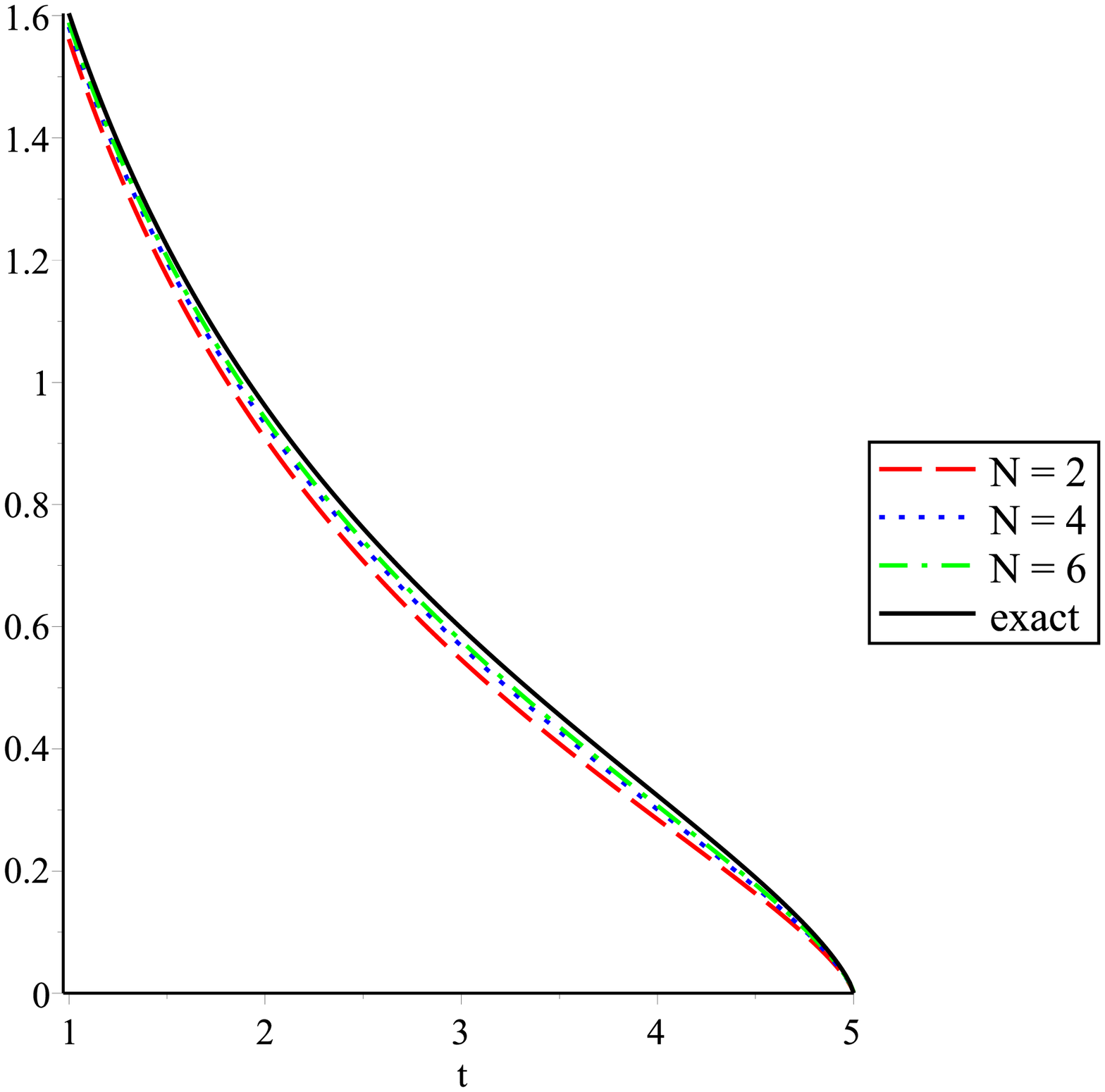}}\hspace{1cm}
\subfigure[Error]{\label{ExpExn}\includegraphics[scale=0.25]{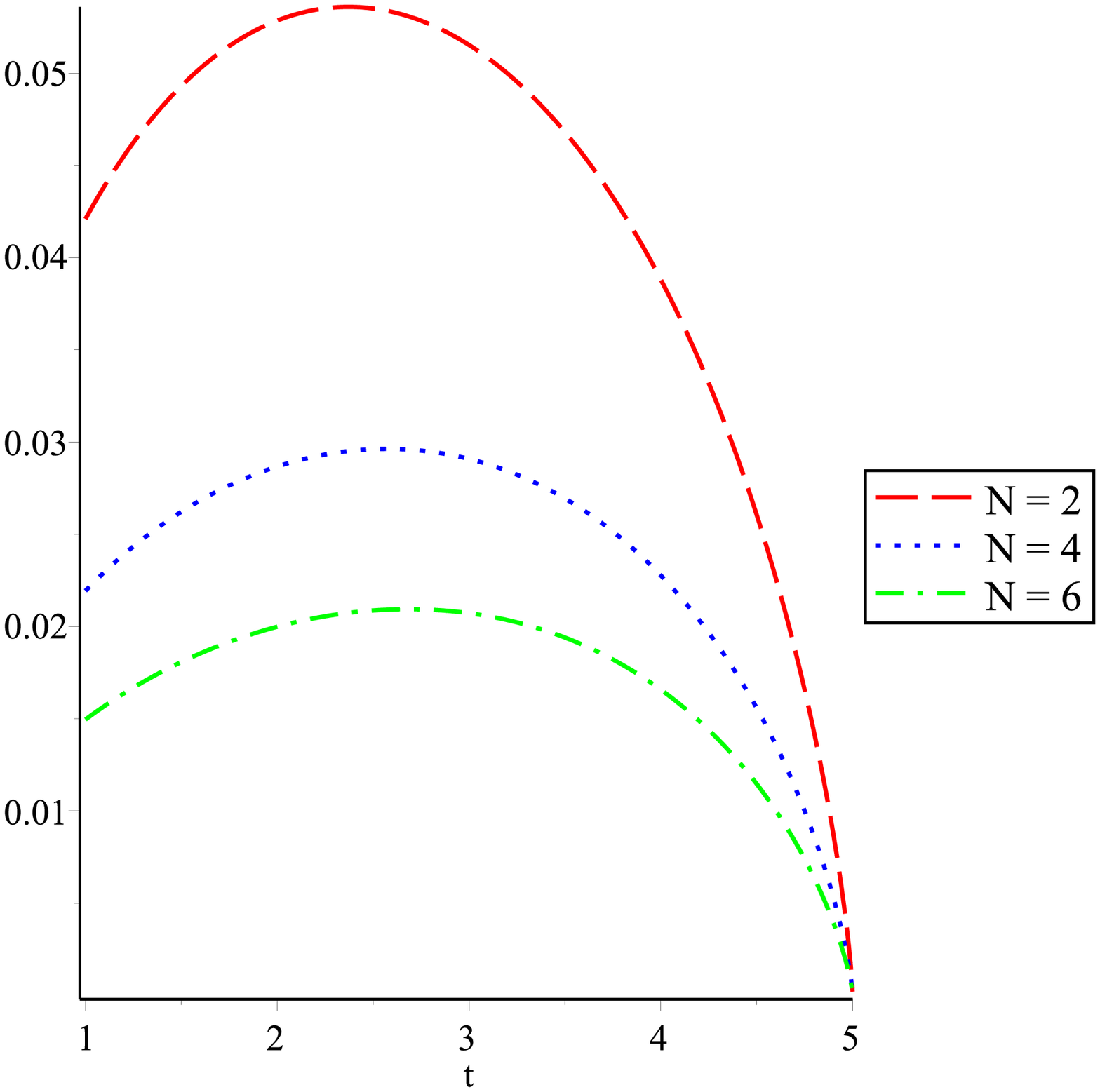}}
\subfigure[${_tD_5^\a}  \overline y(t)$]{\label{PlotI}\includegraphics[scale=0.25]{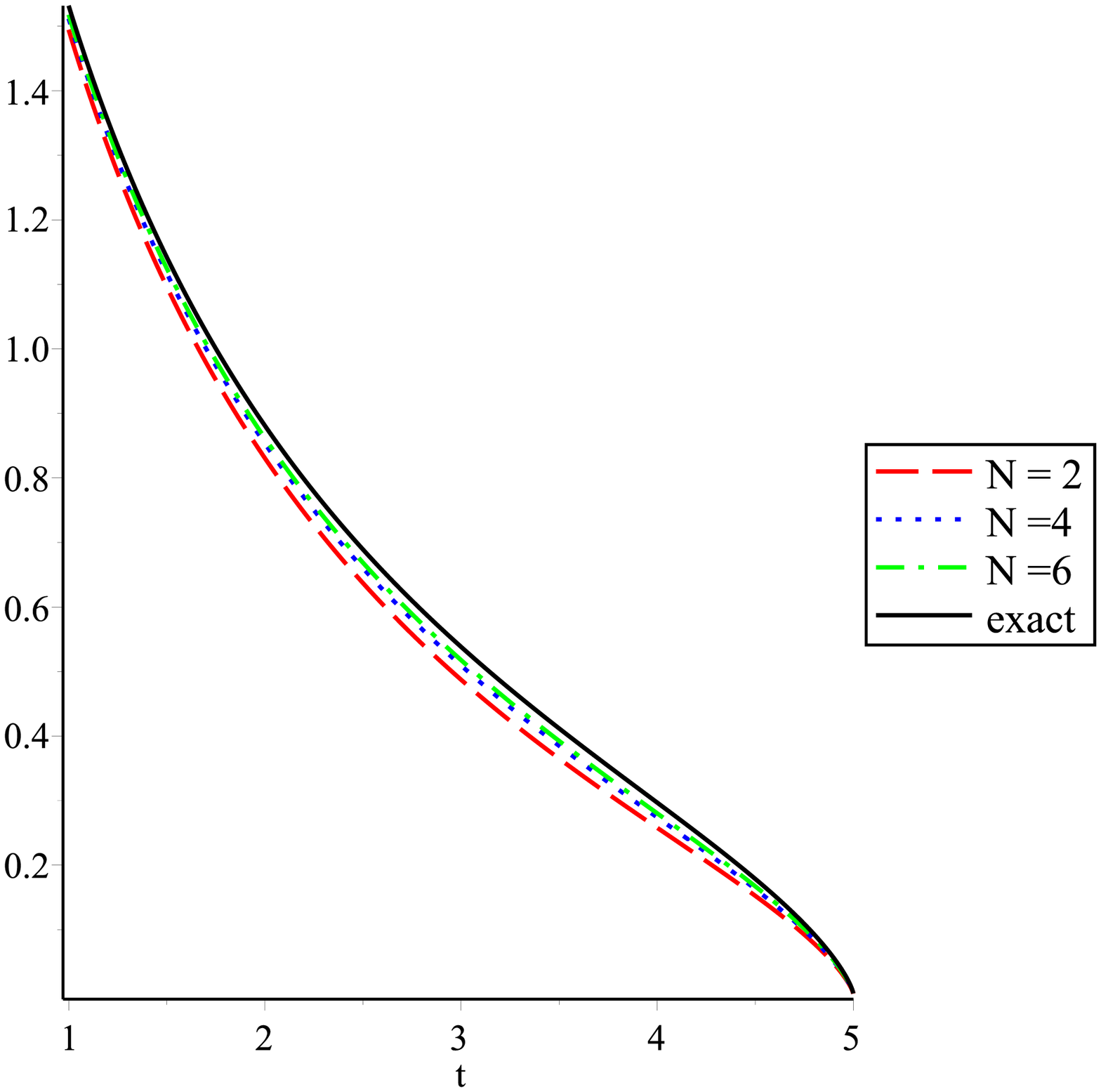}} \hspace{1cm}
\subfigure[Error]{\label{ExpExn}\includegraphics[scale=0.25]{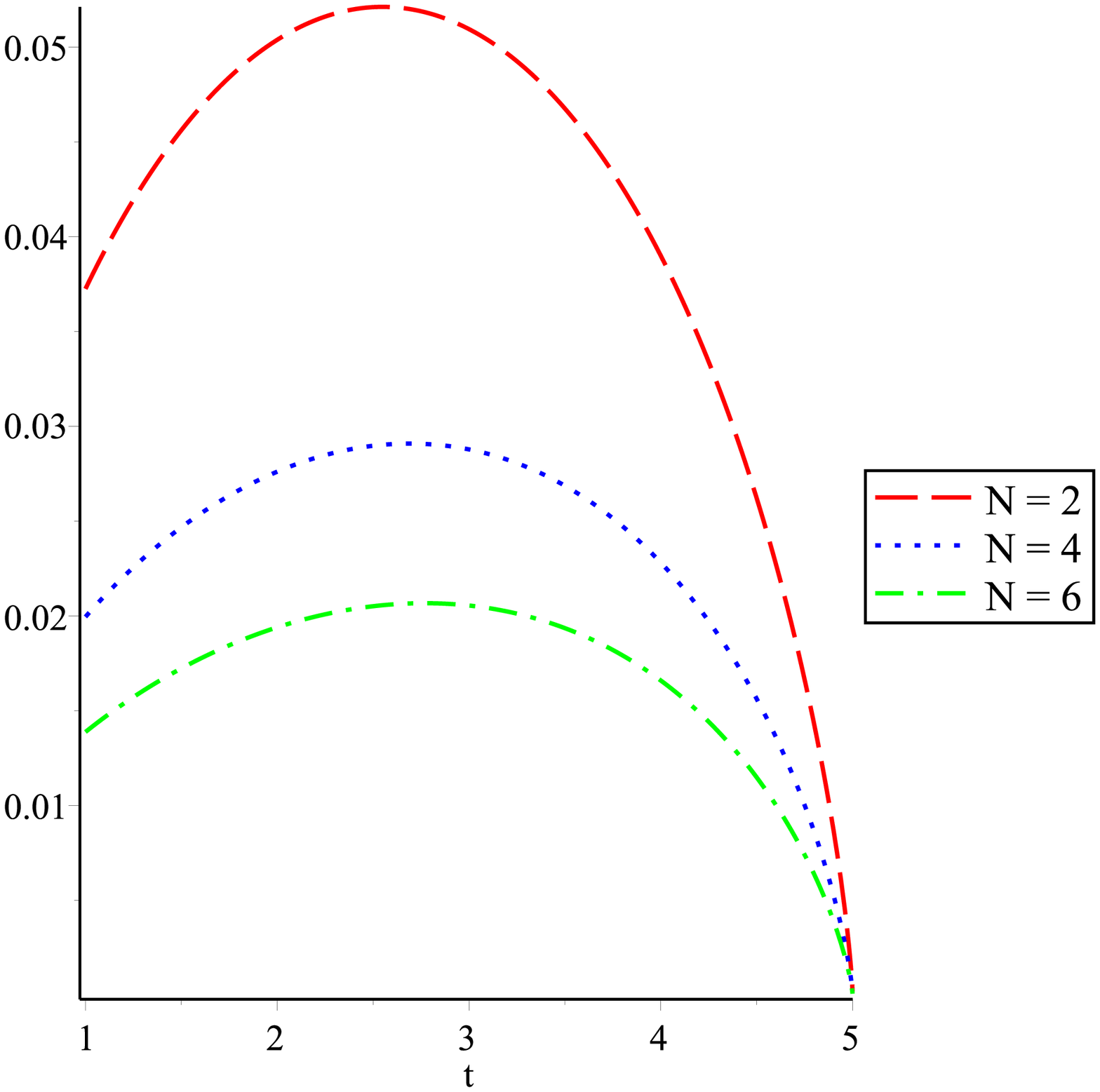}}
\subfigure[${_t\mathcal{D}_5^\a}  \overline y(t)$]{\label{PlotI}\includegraphics[scale=0.25]{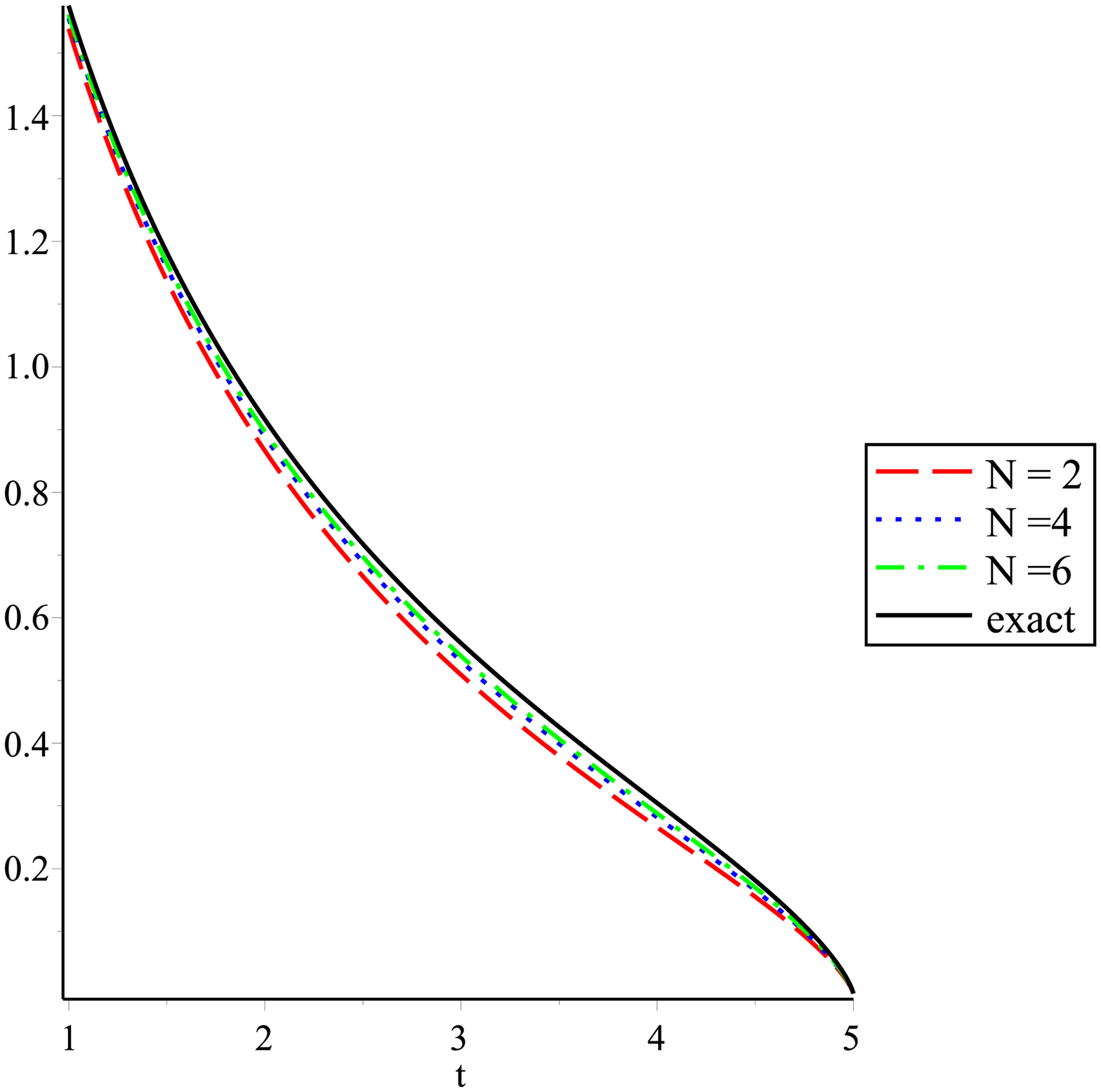}}\hspace{1cm}
\subfigure[Error]{\label{ExpExn}\includegraphics[scale=0.25]{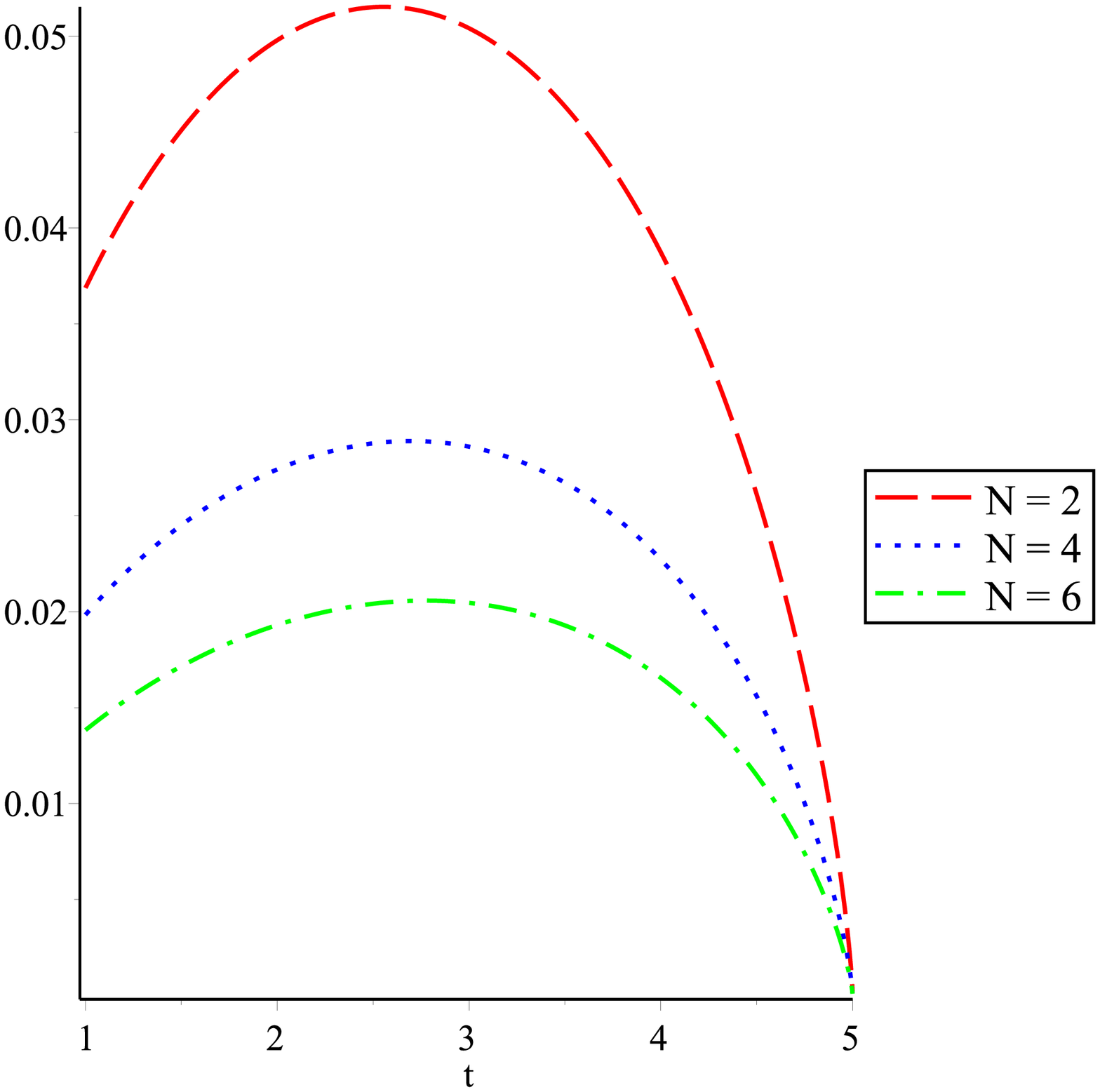}}
\end{center}
\caption{Analytic versus numerical approximations for the right Caputo--Hadamard fractional derivatives.}\label{fig4}
\label{IntExp}
\end{figure}


\section*{Acknowledgments}

This work was supported by Portuguese funds through the CIDMA - Center for Research and Development in Mathematics and Applications, and the Portuguese Foundation for Science and Technology ("FCT–-Funda\c{c}\~{a}o para a Ci\^{e}ncia e a Tecnologia"), within project UID/MAT/04106/2013.



\end{document}